\RequirePackage{fix-cm}

\documentclass[twocolumn]{svjour3}          

\smartqed  

\usepackage{graphicx}
\usepackage{mathptmx}      
\usepackage{amsmath,amssymb,hyperref}

\newcommand{\om}{\omega}
\newcommand{\ds}{\displaystyle}
\newcommand{\La}{\mathcal{L}}

\journalname{}

\begin{document}

\title{
Stability properties of a two-dimensional system involving one Caputo derivative and applications to the investigation of a fractional-order Morris-Lecar neuronal model
\thanks{This work was supported by a grant of the Romanian National Authority for Scientific Research and Innovation, CNCS-UEFISCDI, project no.
PN-II-RU-TE-2014-4-0270.}}

\titlerunning{Stability of a two-dimensional system with one Caputo derivative and applications to a fractional-order Morris-Lecar neuronal model}        

\author{Oana Brandibur \and Eva Kaslik}

\institute{Oana Brandibur and Eva Kaslik \at Institute e-Austria Timisoara\\ Bd. V. Parvan nr. 4, cam. 045B, 300223, Romania \and West University of Timi\c{s}oara\\ Bd. V. Parvan nr. 4, 300223 Romania\\
\email{ekaslik@gmail.com}
}

\date{Received: date / Accepted: date}

\maketitle

\noindent{\textbf{PUBLICATION DETAILS:\\
This paper is now published (in revised form) in\\ Nonlinear Dynamics, 90(4): 2371--2386, 2017.\\
The final publication is available at Springer via \\
http://dx.doi.org/10.1007/s11071-017-3809-2}}

\medskip
\begin{abstract}

Necessary and sufficient conditions are given for the asymptotic stability and instability of a two-dimensional incommensurate order autonomous linear system, which consists of a differential equation with a Caputo-type fractional order derivative and a classical first order differential equation. These conditions are expressed in terms of the elements of the system's matrix, as well as of the fractional order of the Caputo derivative. In this setting, we obtain a generalization of the well known Routh-Hurwitz conditions. These theoretical results are then applied to the analysis of a two-dimensional fractional-order Morris-Lecar neuronal model, focusing on stability and instability properties. This fractional order model is built up taking into account the dimensional consistency of the resulting system of differential equations. The occurrence of Hopf bifurcations is also discussed. Numerical simulations exemplify the theoretical results, revealing rich spiking behavior. The obtained results are also compared to similar ones obtained for the classical integer-order Morris-Lecar neuronal model.

\end{abstract}

\section{Introduction}

In many real world applications, generalizations of dynamical systems using
fractional-order differential equations instead of classical integer-order differential equations have
proven to be more accurate, as fractional-order derivatives provide a good tool for the
description of memory and hereditary properties. Phenomenological description of colored noise \cite{Cottone}, electromagnetic
waves \cite{Engheia}, diffusion and wave propagation \cite{Henry_Wearne,Metzler}, viscoelastic liquids \cite{Heymans_Bauwens}, fractional kinetics \cite{Mainardi_1996} and hereditary effects in nonlinear acoustic waves \cite{Sugimoto} are just a few areas where fractional-order derivatives have been successfully applied.

In addition to straightforward similarities that can be drawn between fractional- and integer-order derivatives and corresponding dynamical systems, it is important to realize that qualitative differences may also arise. For instance, the fractional-order derivative of a non-constant periodic function cannot be a periodic function of the same period \cite{Kaslik2012non}, which is in contrast with the integer-order case. As a consequence, periodic solutions do not exist in a wide class of fractional-order systems. Due to these qualitative differences, which cannot be addressed by simple generalizations of the properties that are available in the integer-order case, the theory of fractional-order systems is a very promising field of research.

With a multitude of practical applications, stability analysis is one of the most important research topics of the qualitative theory of fractional-order systems. Comprehensive surveys of stability properties of fractional differential equations and fractional-order systems have been recently published in \cite{Li-survey,Rivero2013stability}. When it comes to the stability of linear autonomous commensurate fractional order systems, the most important starting point is Matignon's theorem \cite{Matignon}, which has been recently generalized in \cite{Sabatier2012stability}. Linearization theorems (or analogues of the classical Hartman-Grobman theorem) for fractional-order systems have been proved in \cite{Li_Ma_2013,Wang2016stability}. Incommensurate order systems have not received as much attention as their commensurate order counterparts.  Some stability results for linear incommensurate fractional order systems with rational orders have been obtained in  \cite{Petras2008stability}. Oscillations in two-dimensional incommensurate fractional order systems have been investigated in \cite{Datsko2012complex,Radwan2008fractional}.  BIBO stability in systems with irrational transfer functions has been recently investigated in \cite{Trachtler2016}.

In the first of this paper, our aim is to explore necessary and sufficient conditions for the asymptotic stability and instability of a two-dimensional linear incommensurate fractional order system, which consists of a differential equation with a Caputo-type fractional order derivative and a classical first order differential equation.

In the second part of this paper, we propose and analyze a two-dimensional fractional-order Morris-Lecar neuronal model, by replacing the integer-order derivative from the equation describing the dynamics of the membrane potential by a Caputo fractional-order derivative, with careful treatment of the dimensional consistency problem of the resulting system. This fractional-order formulation is justified by experimental results concerning biological neurons \cite{Anastasio}. In \cite{Lundstrom}, it has been underlined that "fractional differentiation provides neurons with a fundamental and general computation ability that can contribute to efficient information processing, stimulus anticipation and frequency-independent phase shifts of oscillatory neuronal firing", emphasizing the importance of developing and analyzing fractional-order models of neuronal activity.

\section{Preliminaries}

\subsection{Fractional-order derivatives}
The Gr\"{u}nwald-Letnikov derivative, the Riemann-Liouville derivative and the Caputo derivative are the most widely used types of fractional derivatives, which are generally not equivalent. In this paper, we restrict our attention to the Caputo derivative, as it is more applicable to real world problems, given that it only requires initial conditions expressed in terms of integer-order derivatives, which represent well-understood features of physical situations. We refer to \cite{Diethelm_book,Kilbas,Lak,Podlubny} for an introduction to fractional calculus and the qualitative analysis of fractional-order dynamical systems.

\begin{definition}
The Caputo fractional-order derivative of an absolutely continuous function $f$ on a real interval $[a,b]$ is
\begin{equation*}
^cD^q f(x)=\frac{1}{\Gamma(1-q)}\int_ 0^x(x-t)^{-q}f^{\prime }(t)dt~,
\end{equation*}
where the gamma function is defined, as usual, as:
\begin{equation*}
\Gamma(z)=\int_0^\infty e^{-t}t^{z-1}dt~.
\end{equation*}
\end{definition}

\begin{remark}
The Caputo derivative of a function $f$ can be expressed as
$$
^c\!D^qf(x)=(k\ast f')(x),
$$
where $k(x)=\frac{x^{-q}}{\Gamma(1-q)}$ and $\ast$ denotes the convolution operation. The Laplace transform of the function $k(x)$ is
$$\mathcal{L}(k)(s)=s^{q-1},$$
where, according to \cite{Doetsch} (example 8 on page 8), the principal value (first branch) of the complex power function has to be taken into account. Therefore, the Laplace transform of the Caputo derivative is deduced in the following way:
\begin{align*}
\mathcal{L}(^c\!D^q f)(s)&=\mathcal{L}(k\ast f')(s)=\mathcal{L}(k)(s)\cdot \mathcal{L}(f')(s)=\\
&=s^{q-1}(s\mathcal{L}(f)(s)-f(0))=\\
&=s^{q}\mathcal{L}(f)(s)-s^{q-1}f(0).
\end{align*}
\end{remark}

In the following, we give an elementary result that will be useful in the theoretical analysis of the Morris-Lecar neuronal model. For completeness, the proof is included in Appendix A.

\begin{proposition}\label{prop_aq}
	Let $f$ and $g$ be two functions such that $g(x)=f(ax)$, with $a\neq 0$.
	Then
	$$^c\!D^qg(x)=a^q\cdot ^c\!\!D^q f(ax)$$
\end{proposition}

\subsection{Stability of fractional-order systems}

Let us consider the  $n$-dimensional fractional-order system
\begin{equation}\label{sys.gen}
^c\!D^\mathbf{q}\mathbf{x}(t)=f(t,\mathbf{x})
\end{equation}
where $\mathbf{q}=(q_1,q_2,...,q_n)\in(0,1)^n$ and $f:[0,\infty)\times\mathbb{R}^n\rightarrow \mathbb{R}^n$ is continuous on the whole domain of definition and Lipschitz-continuous with respect to the second variable, such that
$$f(t,0)=0\quad \textrm{for any }t\geq 0.$$
Let $\varphi(t,x_0)$ denote the unique solution of (\ref{sys.gen}) which satisfies the initial condition $x(0)=x_0$ (it is important to note that the conditions on the function $f$ given above guarantee the existence and uniqueness of such a solution \cite{Diethelm_book}).

It is well-known that in general, the asymptotic stability of the trivial solution of system (\ref{sys.gen}) is not of exponential type \cite{Cermak2015,Gorenflo_Mainardi}, because of the presence of the memory effect. Due to this observation, a special type of non-exponential asymptotic stability concept has been defined for fractional-order differential equations \cite{Li_Chen_Podlubny}, called Mittag-Leffler stability.  In this paper, we are concerned with  $\mathcal{O}(t^{-\alpha})$-asymptotical stability, which reflects the algebraic decay of the solutions.

\begin{definition}\label{def.stability}
The trivial solution of (\ref{sys.gen}) is called \emph{stable} if for any $\varepsilon>0$
there exists $\delta=\delta(\varepsilon)>0$ such that for every $x_0\in\mathbb{R}^n$ satisfying $\|x_0\|<\delta$ we have
$\|\varphi(t,x_0)\|\leq\varepsilon$ for any $t\geq 0$.

The trivial solution  of (\ref{sys.gen}) is called \emph{asymptotically stable} if it is stable and there
exists $\rho>0$ such that $\lim\limits_{t\rightarrow\infty}\varphi(t,x_0)=0$ whenever $\|x_0\|<\rho$.

Let $\alpha>0$. The trivial solution  of (\ref{sys.gen}) is called \emph{$\mathcal{O}(t^{-\alpha})$-asymptotically stable} if it is stable and there exists $\rho>0$ such that for any $\|x_0\|<\rho$ one has:
$$\|\varphi(t,x_0)\|=\mathcal{O}(t^{-\alpha})\quad\textrm{as }t\rightarrow\infty.$$
\end{definition}

It is important to remark that $\mathcal{O}(t^{-\alpha})$-asymptotic stability, as defined above, clearly implies asymptotic stability.

\section{Stability results for a linear system involving one Caputo derivative}

We will first investigate the stability properties of the following linear system:
\begin{equation}\label{linearsys}
\begin{bmatrix}
^c\!D^qx(t)\\\dot{y}(t)
\end{bmatrix}=A\cdot \begin{bmatrix}
x(t)\\y(t)
\end{bmatrix}
\end{equation}
where $A=(a_{ij})$ is a real 2-dimensional matrix and $q\in(0,1)$.

\begin{remark}
We may assume $a_{12}\neq 0$. Otherwise, the first equation of system (\ref{linearsys}) would be decoupled from the second equation.
\end{remark}

Applying the Laplace transform to system (\ref{linearsys}), we obtain:
\begin{equation*}
\begin{bmatrix}
s^qX(s)-s^{q-1}x(0)\\sY(s)-y(0)
\end{bmatrix}=A\cdot \begin{bmatrix}
X(s)\\Y(s)
\end{bmatrix},
\end{equation*}
where $X(s)=\La(x)(s)$ and $Y(s)=\La(y)(s)$ denote the Laplace transforms of the functions $x$ and $y$ respectively, and $s^q$ represents the principal value (first branch) of the complex power function \cite{Doetsch}. Therefore:
$$\left(\text{diag}(s^q,s)-A\right)\cdot\begin{bmatrix}
X(s)\\Y(s)
\end{bmatrix}=\begin{bmatrix}
s^{q-1} x(0)\\ y(0)
\end{bmatrix}$$
In the following, we will denote:
\begin{align*}
\Delta_A(s)&=\det\left(\text{diag}(s^q,s)-A\right)=\\
&=s^{q+1}-a_{11}s-a_{22}s^q+\det(A).
\end{align*}
We can easily express
\begin{align}\label{transfer.functions}
\nonumber X(s)&=\frac{x(0)s^{q}(s-a_{22})+a_{12}y(0)s}{s\Delta_A(s)}\\
Y(s)&=\frac{a_{21}x(0)s^q+y(0)s(s^q-a_{11})}{s\Delta_A(s)}
\end{align}

With the aim of proving sufficient conditions for the stability of linear systems of fractional differential equations, several authors have exploited the Final Value Theorem of the Laplace Transform \cite{Deng_2007,Odibat2010}. For the sake of completeness, we state the following result:

\begin{theorem}\label{thm.lin.stab}
\begin{enumerate}
\item System (\ref{linearsys}) is $\mathcal{O}(t^{-q})$-globally asymptotically stable if and only if all the roots of $\Delta_A(s)$ are in the open left half-plane ($\Re(s)<0$).
\item If $\det(A)\neq 0$ and $\Delta_A(s)$ has at least one root in the open right half-plane ($\Re(s)>0$), then system (\ref{linearsys}) is unstable.
\end{enumerate}
\end{theorem}

\begin{proof}
\emph{Part 1 - Necessity.} Assume that system (\ref{linearsys}) is $\mathcal{O}(t^{-q})$-globally asymptotically stable and let $(x(t),y(t))$ denote the solution of system (\ref{linearsys}) which satisfies the initial condition $(x(0),y(0))=(x_0,y_0)\in\mathbb{R}^2$. We may choose $x_0,y_0\neq 0$. It follows that there exist $M>0$ and $T>0$ such that
$$|x(t)|\leq\|(x(t),y(t))\|\leq Mt^{-q}\quad\textrm{for any }t\geq T.$$
We obtain that the Laplace transform $X(s)$ is absolutely continuous and holomorphic in the open right half-plane ($\Re(s)>0$) \cite{Doetsch}. Therefore, $X(s)$ does not have any poles in the open right half-plane. From (\ref{transfer.functions}), we remark that
$$X(s)=\frac{x_0s^{q}(s-a_{22})+a_{12}y_0s}{s\Delta_A(s)}$$
and the function from the numerator is holomorphic on $\mathbb{C}\setminus\mathbb{R}_{-}$. So far, we have obtained:
$$\Delta_A(s)\neq 0\quad\textrm{for any }s\in\mathbb{C},~\Re(s)>0.$$
We now argue that $\Delta_A(0)\neq 0$. Indeed, assuming that $\Delta_A(0)=0$, it follows that $\det(A)=0$ and
$$\Delta_A(s)=s^{q+1}-a_{11}s-a_{22}s^q.$$
Therefore:
\begin{align*}
\lim\limits_{s\rightarrow 0}sX(s)&=\lim\limits_{s\rightarrow 0}\frac{x_0s^{q}(s-a_{22})+a_{12}y_0s}{\Delta_A(s)}=\\
&=\lim\limits_{s\rightarrow 0}\frac{x_0(s-a_{22})+a_{12}y_0s^{1-q}}{s-a_{11}s^{1-q}-a_{22}}=\\
&=\left\{
    \begin{array}{ll}
      x_0, & \textrm{if }a_{22}\neq 0 \\
      -\frac{a_{12}}{a_{11}}y_0, & \textrm{if }a_{22}=0
    \end{array}
  \right.\quad\neq 0,
\end{align*}
which contradicts the Final Value Theorem for the Laplace transform $X(s)$ (since $x(t)\rightarrow 0$ as $t\rightarrow\infty$). Hence, $\Delta_A(0)\neq 0$.

Now  and consider the solution $(x(t),y(t))$  of system (\ref{linearsys}) which satisfies the initial condition $(x(0),y(0))=(0,\frac{1}{a_{12}})$. For $x(t)$ we obtain the Laplace transform $X(s)=\Delta_A(s)^{-1}$. Assuming that $\Delta_A(s)$ has a root on the imaginary axis (but not at the origin), it follows that $X(s)$ has a pole on the imaginary axis, which implies that $x(t)$ has persistent oscillations, contradicting the convergence of $x(t)$ to the limit $0$, as $t\rightarrow\infty$.

Therefore, we obtain $\Delta_A(s)\neq 0$, for any $s\in\mathbb{C}$, $\Re(s)\geq 0$.

\emph{Part 1 - Sufficiency.} Let $(x(t),y(t))$ denote the solution of system (\ref{linearsys}) which satisfies the initial condition $(x(0),y(0))=(x_0,y_0)\in\mathbb{R}^2$. Assuming that all the roots of $\Delta_A(s)$ are in the open left half-plane, it follows that all the poles of the Laplace transforms functions $X(s)$ and $Y(s)$ given by (\ref{transfer.functions}) are either in the open left half-plane or at the origin, and $X(s)$ and $Y(s)$ have at most a single pole at the origin (in fact, only $X(s)$ has a simple pole at the origin). A simple application of the Final Value Theorem of the Laplace transform \cite{Chen_Lundberg} yields
\begin{align*}
\lim_{t\rightarrow\infty}x(t)&=\lim_{s\rightarrow 0}sX(s)=\lim_{s\rightarrow 0}\frac{x_0s^{q}(s-a_{22})+a_{12}y_0s}{\Delta_A(s)}=0;\\
\lim_{t\rightarrow\infty}y(t)&=\lim_{s\rightarrow 0}sY(s)=\lim_{s\rightarrow 0}\frac{a_{21}x_0s^{q}+y_0s(s^q-a_{11})}{\Delta_A(s)}=0.
\end{align*}
Moreover, the Laplace transform $X(s)$ is holomorphic in the left half-plane, except at the origin and has the asymptotic expansion
$$X(s)\sim \sum_{n=0}^\infty c_n s^{\lambda_n},\quad \textrm{as } s\rightarrow 0, $$
where $\lambda_0=q-1<\lambda_1<...<\lambda_n<...$. Using Theorem 37.1 from \cite{Doetsch}, the following asymptotic expansion is obtained:
$$x(t)\sim \sum_{n=0}^\infty \frac{c_n}{\Gamma(-\lambda_n)}\frac{1}{t^{\lambda_n+1}} ,\quad \textrm{as } t\rightarrow \infty, $$
where $\Gamma$ represents the Gamma function with the understanding that
$$\frac{1}{\Gamma(-\lambda_n)}=0\quad \textrm{if }\lambda_n\in\mathbb{Z}_+.$$
As $\lambda_0+1=q$, it follows that $x(t)$ converges to $0$ as $t^{-q}$.

On the other hand, the Laplace transform $Y(s)$ is holomorphic in the whole left half-plane and has a similar asymptotic expansion as $X(s)$. As above, it follows that $y(t)$ converges to $0$ as $t^{-q}$.

Combining the convergence results for the two components $x(t)$ and $y(t)$, it follows, based on Definition \ref{def.stability} that system (\ref{linearsys}) is $\mathcal{O}(t^{-q})$-globally asymptotically stable.

\emph{Part 2.} Assume that $\det(A)\neq 0$, which is equivalent to $\Delta_A(0)\neq 0$. Consider the solution of $(x(t),y(t))$  of system (\ref{linearsys}) which satisfies the initial condition $(x(0),y(0))=(0,y_0)$, with an arbitrary $y_0\in\mathbb{R}^\star$. This solution has the Laplace transform $X(s)=a_{12}y_0\Delta_A(s)^{-1}$. Based on Proposition 3.1 from \cite{Bonnet_2002}, it follows that $\Delta_A(s)$ has a finite number of roots in $\mathbb{C}\setminus\mathbb{R}_{-}$, and in particular, in the open right half-plane. Obviously, the Laplace transform $X(s)$ is analytic in $\mathbb{C}\setminus\mathbb{R}_{-}$, except at the poles given by the roots of $\Delta_A(s)$.

If $\Delta_A(s)$ has at least one root in the open right half-plane, let us denote by $\rho>0$ the real part of a dominant pole of $X(s)$, i.e. $\rho=\max\{\Re(s):\Delta_A(s)=0\}$, and by $\nu\geq 1$ the largest order of a dominant pole. Following Theorem 35.1 from \cite{Doetsch}, we obtain that $|x(t)|$ is asymptotically equal to $k~t^{\nu-1} e^{\rho t}$ (with $k>0$) as $t\rightarrow \infty$. Hence, $x(t)$ is unbounded and therefore, system (\ref{linearsys}) is unstable.
\qed\end{proof}

Taking into account the special form of the characteristic function $\Delta_A(s)$ given above, we prove the following result:

\begin{proposition}\label{prop122}
Consider the complex-valued function
	$$\Delta (s)=s^{q+1}+as+bs^q+c,$$
where $q\in (0,1)$, $a,b,c\in\mathbb{R}$, $b>0$, and $s^q$ represents the principal value (first branch) of the complex power function.
\begin{enumerate}
\item If $c<0$, then $\Delta(s)$ has at least one positive real root.
\item $\Delta(0)=0$ if and only if $c=0$.
\item Assume $c>0$.
	\begin{enumerate}
		\item If $a\geq 0$ then all roots of $\Delta(s)$ satisfy $\Re(s)<0$.
		\item The function $\Delta(s)$ has a pair of pure imaginary roots if and only if
\begin{align}\label{eq.a.star}
a&=a^\star(b,c,q)=\\		\nonumber &=\!\!-b^q\!\left(\!h_q^{-1}\Big(\frac{c}{b^{q+1}}\Big)\!\right)^{\!\!q-1}\!\!\!\!\Big(\!h_q^{-1}\!\!\Big(\!\frac{c}{b^{q+1}}\!\Big)\cos\!\frac{q\pi}{2}+\sin\!\frac{q\pi}{2}\Big)
\end{align}
where $h_q:\Big(\cot\frac{q\pi}{2},\infty\Big)\to(0,\infty)$ is the function defined by
$$h_q(\omega)=\omega^q\Big(\omega\sin\frac{q\pi}{2}-\cos\frac{q\pi}{2}\Big).$$
\item If $s(a,b,c,q)$ is a root of $\Delta(s)$ such that
$$\Re(s(a^\star,b,c,q))=0,$$ where $a^\star=a^\star(b,c,q)$, the following transversality condition holds:
    $$\frac{\partial \Re(s)}{\partial a}\Big|_{a=a^*}<0.$$
\item All roots of $\Delta(s)$ are in the left half-plane if and only if $a>a^\star(b,c,q)$.
\item $\Delta(s)$ has a pair of roots in the right half-plane if and only if $a<a^\star(b,c,q)$.
\item For any $q\in(0,1)$, the following inequalities hold:
$$a^\star(b,c,q)\leq -b^q\leq -\min\{b,1\}.$$
\end{enumerate}
\end{enumerate}
\end{proposition}

\begin{proof}$ $\\
1. We have $\Delta(0)=c<0$ and $\Delta(\infty)=\infty$, and therefore, due to the continuity of the function $\Delta(s)$ on $(0,\infty)$, it follows that it has at least one positive real root.\\

\noindent 2. The proof is trivial as $\Delta(0)=c$.\\

\noindent 3.(a) Let $a\geq 0$ and $c>0$. Assuming, by contradiction, that $\Delta(s)$ has a root $s$ with $\Re(s)\geq 0$, it follows that
    $$|\arg(s)|\leq\frac{\pi}{2}\Rightarrow |\arg(s^q)|=q\cdot |\arg(s)|\leq\frac{q\pi}{2}<\frac{\pi}{2}.$$
Therefore, $\Re(s^q)>0$. On the other hand, we have:
	\begin{align*}
	s^{q+1}+as+bs^q+c=0~~\Leftrightarrow~~ s^q=\frac{-as-c}{s+b}
	\end{align*}
	and hence
	\begin{align*} \Re(s^q)&=\Re\Big(\frac{-as-c}{s+b}\Big)=\Re\left[\frac{(-as-c)(\bar{s}+b)}{|s+b|^2}\right]\\
&=\frac{\Re\big[(-as-c)(\bar{s}+b)\big]}{|s+b|^2}\\
	&=\frac{\Re(-as-c)\Re(\bar{s}+b)-\Im(-as-c)\Im(\bar{s}+b)}{|s+b|^2}\\
	&=\frac{(-a\Re(s)-c)(\Re(s)+b)+a\Im(s)(-\Im(s))}{|s+b|^2}\\
	&=\frac{-a|s|^2-(ab+c)\Re(s)-bc }{|s+b|^2}.
	\end{align*}
As $a\geq 0$, $b>0$, $c\geq 0$, then $-a|s|^2-(ab+c)\Re(s)-bc\leq 0$ and so $\Re(s^q)\leq 0$, which contradicts $\Re(s^q)>0$. We conclude that the equation $\Delta(s)=0$ does not have any roots with $\Re(s)\geq 0$.\\
	
\noindent 3.(b) Let $a<0$ and $c>0$. Assuming that $\Delta(s)$ has a pair of pure imaginary roots, there exists $\omega>0$ such that $s=ib\omega$ is a root of $\Delta(s)$. From $\Delta(ib\omega)=0$ we have:
\begin{align*}
\label{unu}
\nonumber b^{q+1}&\omega^{q+1}\Big(-\sin \frac{q\pi}{2}+i\cos\frac{q\pi}{2}\Big)+iab\omega+\\
&+b^{q+1}\Big(\cos \frac{q\pi}{2}+i\sin \frac{q\pi}{2}\Big)\omega^q+c=0
\end{align*}
Taking the real and the imaginary parts in this equation, we obtain:
	\begin{align*}
	&  -b^{q+1}\om^{q+1}\sin \frac{q\pi}{2}+b^{q+1}\omega^q\cos \frac{q\pi}{2}+c=0\\
	& b^{q+1}\omega^{q+1}\cos \frac{q\pi}{2}+ab\omega+b^{q+1}\omega^q\sin \frac{q\pi}{2}=0
	\end{align*}
	which is equivalent to
	\begin{equation}\label{doi}
	\begin{cases}
	& a= -b^{q}\omega^{q-1}\Big(\omega\cos \frac{q\pi}{2}+\sin \frac{q\pi}{2}\Big)\\
	& c=b^{q+1}\omega^q\Big(\omega\sin \frac{q\pi}{2}-\cos \frac{q\pi}{2}\Big)
	\end{cases}
	\end{equation}
	As $c>0$, it results that $\omega\sin \frac{q\pi}{2}>\cos \frac{q\pi}{2}$, which leads to $\omega>\cot\frac{q\pi}{2}$. Since
	\begin{equation}\label{treii}
	\frac{c}{b^{q+1}}=\omega^q\Big(\omega\sin \frac{q\pi}{2}-\cos \frac{q\pi}{2}\Big),
	\end{equation}
	we consider the function $h_q:\Big(\cot \frac{q\pi}{2},\infty\Big)\to(0,\infty)$ defined by
$$h_q(\omega)=\omega^q\Big(\omega\sin \frac{q\pi}{2}-\cos \frac{q\pi}{2}\Big).$$
It is easy to see that for any $\omega> \cot\frac{q\pi}{2}$, we have:
$${h_q}'(\omega)=q\omega^{q-1}\Big(\omega\sin \frac{q\pi}{2}-\cos \frac{q\pi}{2}\Big)+\omega^q\sin \frac{q\pi}{2}> 0. $$
Hence, $h_q$ is an increasing function on the interval $\Big(\cot \frac{q\pi}{2},\infty\Big)$ and therefore $h_q$ is invertible, with the inverse denoted by $h_q^{-1}:(0,\infty)\to (\cot \frac{q\pi}{2},\infty)$. Hence, from (\ref{treii}) we obtain:
	$$\omega=h_q^{-1}\Big(\frac{c}{b^{q+1}}\Big).$$
From the first equation of (\ref{doi}), we conclude that $a=a^\star(b,c,q)$.\\
	
\noindent 3.(c) Let $s(a,b,c,q)$ denote the root of $\Delta(s)$ with the property $$s(a^\star,b,c,q)=ib\omega,$$ as in 3.(b), where $a^\star=a(b,c,q)$.
Differentiating with respect to $a$ in the equation:
$$s^{q+1}+as+bs^q+c=0$$
we obtain:
$$(q+1)s^q\cdot \frac{\partial s}{\partial a}+s+a\cdot \frac{\partial s}{\partial a}+b\cdot q\cdot s^{q-1}\cdot \frac{\partial s}{\partial a}=0,$$
and hence:
$$
	\frac{\partial s}{\partial a}=\frac{-s}{(q+1)s^q+qbs^{q-1}+a}.
$$
Taking the real part of this equation, we obtain:
$$
	\frac{\partial \Re(s)}{\partial a}=\Re\Big(\frac{\partial s}{\partial a}\Big)=\Re\Big(\frac{-s}{(q+1)s^q+qbs^{q-1}+a}\Big).
$$
Therefore:
$$\frac{\partial \Re(s)}{\partial a}\Big|_{a=a^\star}=\Re\Big(\frac{-ib\omega}{(q+1)(ib\omega)^q+qb(ib\omega)^{q-1}+a^\star}\Big).$$
Denoting $P(\omega)=(q+1)(ib\omega)^q+qb(ib\omega)^{q-1}+a^\star$, we obtain:
\begin{align*}
\frac{\partial \Re(s)}{\partial a}\Big|_{a=a^\star}&=\Re\Big( \frac{-ib\omega}{P(\omega)}\Big)=b\omega\cdot\Re\Big(\frac{-i\overline{P(\omega)}}{|P(\omega)|^2}\Big)\\
&=\frac{b\omega}{|P(\omega)|^2}\cdot(-\Im (P(\omega)))
\end{align*}
As
\begin{align*}
P(\omega)=&(q+1)b^q\omega^q\Big(\cos \frac{q\pi}{2}+i\sin \frac{q\pi}{2}\Big)+\\
&+qb^q\omega^{q-1}(-i)\Big(\cos \frac{q\pi}{2}+i\sin \frac{q\pi}{2}\Big)+a^\star,
\end{align*}
we have:
$$\Im (P(\omega))=(q+1)b^q\omega^q\sin \frac{q\pi}{2}-qb^q\omega^{q-1}\cos \frac{q\pi}{2}.$$
As $\om>\cot\left(\frac{q\pi}{2}\right)$, it results that
$$\frac{\partial \Re(s)}{\partial a}\Big|_{a=a^\star}=\frac{b^{q+1}\omega^q}{|P(\omega)|^2}\Big(q\cos \frac{q\pi}{2}-(q+1)\omega\sin\frac{q\pi}{2}\Big)<0.$$

\noindent 3.(d,e) The transversality condition obtained above shows that $\Re(s)$ is decreasing in a neighborhood of $a^\star$, and therefore, when $a$ decreases below the critical value $a^\star=a^\star(b,c,q)$, the pair of complex conjugated roots $(s,\overline{s})$ cross the imaginary axis from the left half-plane into the right half-plane. Combined with 3.(a), we obtain the desired conclusions.\\

\noindent 3.(f) We will first prove that for any $q\in(0,1)$ and any $\omega>0$, the following inequality holds:
\begin{equation}\label{ineq.q.w}
\omega\cos\frac{q\pi}{2}+\sin\frac{q\pi}{2}\geq \omega^{1-q}.
\end{equation}
Indeed, let us denote $\theta=\arctan\omega\in\left(0,\frac{\pi}{2}\right)$. Inequality (\ref{ineq.q.w}) is equivalent to
$$\tan\theta\cos\frac{q\pi}{2}+\sin\frac{q\pi}{2}\geq (\tan\theta)^{1-q}$$
which can be rewritten as
$$\sin\left(\theta+\frac{q\pi}{2}\right)\geq (\sin\theta)^{1-q}(\cos\theta)^q.$$
As the natural logarithm is an increasing function and $\cos(\theta)=\sin\left(\theta+\frac{\pi}{2}\right)$, this inequality is further equivalent to:
$$
\ln\left(\sin\left(\theta+\frac{q\pi}{2}\right)\right)\geq (1-q)\ln\left(\sin\theta\right)+q\ln\left(\sin\left(\theta+\frac{\pi}{2}\right)\right).
$$
This last inequality easily follows from the fact that the function $f(x)=\ln(\sin(x))$, defined on the interval $(0,\pi)$, is a concave function (Jensen's inequality).

Therefore, inequality (\ref{ineq.q.w}) holds and based on the definition of $a^\star(b,c,q)$, it leads to $a^\star(b,c,q)\leq-b^q$, for any $b>0$, $c>0$ and $q\in(0,1)$.

The second inequality easily follows from the properties of the function $b^x$, where $b>0$ and $x\in[0,1]$. If $b\in(0,1)$ then $b^x$ is decreasing on $[0,1]$ and therefore, for any $q\in[0,1]$, we have $b^q\geq b^1=b=\min\{b,1\}$. On the other hand, if $b\geq 1$, then $b^x$ is increasing on $[0,1]$ and hence, for any $q\in[0,1]$, we obtain $b^q\geq b^0=1=\min\{b,1\}$. We conclude that $b^q\geq\min\{b,1\}$, for any $b>0$ and $q\in[0,1]$.\qed\end{proof}

With the aim of deducing sufficient stability conditions which do not depend on the fractional order $q$, we state the following:

\begin{proposition}\label{prop.indep.q} Let $b>0$, $c>0$, and consider the complex-valued function $\Delta(s)$ defined as in Proposition \ref{prop122}.
\begin{enumerate}
		\item If $a>-\min\{b,1\}$ then all roots of $\Delta(s)$ are in the open left-half plane, regardless of $q$.
		\item Let $a\leq -\min\{b,1\}$. If either of the following conditions hold
		\begin{enumerate}
			\item $a+b+c+1\leq 0$;
			\item $0<a+b+c+1<(\sqrt{c}-1)^2$ and $c>1$;
		\end{enumerate}
then $\Delta(s)$ has at least one positive real root, regardless of $q$.
\end{enumerate}
\end{proposition}

\begin{proof}
1.  Let us consider an arbitrary $q\in(0,1)$ and $a>-\min\{b,1\}$. From Proposition \ref{prop122}.(f) we have
$$a> -\min\{b,1\}\geq -b^q\geq a^\star(b,c,q).$$
Hence, based on Proposition \ref{prop122}.(d) it follows that all the roots of $\Delta(s)$ are in the open left half-plane.
 			
\noindent 2. We have:
		$$\Delta(s)\leq \begin{cases}
			& (1+a)s+b+c, \quad \textrm{if }s\in(0,1),\\
			& s^2+(a+b)s+c,\quad \textrm{if } s\geq 1,
			\end{cases}$$
			and we denote $(1+a)s+b+c=p_1(s)$ and $s^2+(a+b)s+c=p_2(s)$. It is easy to see that if either $p_1$ or $p_2$ have positive real roots, so does $\Delta(s)$.

\noindent (a) If $a+b+c+1\leq 0$ then $\Delta(1)=1+a+b+c\leq 0$ and $\Delta(\infty)=\infty$. So $\Delta(s)$ has at least one real root belonging to the interval $[1,\infty)$.
				
\noindent (b) If $a+b+c+1>0$ then $p_1(s)>0$ for every $s\in(0,1)$. Since $c>0$, elementary calculus shows that necessary and sufficient conditions for $p_2(s)$ to take negative values in a subinterval of $[1,\infty)$ are: discriminant $\delta=(a+b)^2-4c\geq 0$ and $-\frac{a+b}{2}>1$ (i.e. the minimum point of $p_2$ is larger than $1$). In turn, these are equivalent to $c>1$ and $-(c+1)<a+b<-2\sqrt{c}$, which lead to the desired conclusion.
\qed\end{proof}

Based on Theorem \ref{thm.lin.stab} and Propositions \ref{prop122} and \ref{prop.indep.q}, we obtain the following conditions for the stability of system (\ref{linearsys}), with respect to its coefficients and the fractional order $q$:

\begin{corollary}\label{cor.stab.lin}
Consider the linear system (\ref{linearsys}) with the fractional order $q\in(0,1)$. Denoting $a=-a_{11}$, $b=-a_{22}$, $c=\det(A)$ and assuming $b>0$, it follows that:
\begin{enumerate}
\item If $c<0$, system (\ref{linearsys}) is unstable, regardless of the fractional order $q$.
\item Assume that $c>0$.
\begin{itemize}
\item[(a)] System (\ref{linearsys}) is $\mathcal{O}(t^{-q})$-asymptotically stable if and only if $a>a^\star(b,c,q)$, where $a^\star(b,c,q)$ is defined by (\ref{eq.a.star}).
\item[(b)] If $a>-\min\{b,1\}$, system (\ref{linearsys}) is asymptotically stable, regardless of the fractional order $q$.
\item[(c)] System (\ref{linearsys}) is unstable if $a<a^\star(b,c,q)$, where $a^\star(b,c,q)$ is defined by (\ref{eq.a.star}).
\item[(d)] If $a\leq -\min\{b,1\}$ and either of the following conditions hold
		\begin{itemize}
			\item $a+b+c+1\leq 0$;
			\item $0<a+b+c+1<(\sqrt{c}-1)^2$ and $c>1$;
		\end{itemize}
then system (\ref{linearsys}) is unstable, regardless of the fractional order $q$.
\end{itemize}
\end{enumerate}
\end{corollary}

\begin{remark}
Condition 2.(a) from Corollary \ref{cor.stab.lin} is a generalization of the well known Routh-Hurwitz conditions for two dimensional systems of first order differential equations. Indeed, if $A$ is the system's matrix, the Routh-Hurwitz conditions provide that the (first-order) system is asymptotically (exponentially) stable if and only if $\text{trace}(A)<0$ and $\det(A)>0$. In our setting, it can easily be seen that for $q=1$, equation (\ref{eq.a.star}) gives $a^\star(b,c,1)=-b$. Hence, with the notations $a=-a_{11}$, $b=-a_{22}$ and $c=\det(A)$, condition 2.(a) for the particular case $q=1$ is equivalent to the Routh-Hurwitz conditions.
\end{remark}

\begin{remark}
Throughout this section, we have considered the assumption $b=-a_{22}>0$. A similar lengthy reasoning can also be applied in the case $b\leq 0$, to deduce necessary and sufficient conditions for the asymptotic stability or instability of system (\ref{linearsys}). However, as we will see in the following section, in the stability analysis of the equilibrium states of a fractional-order Morris-Lecar neuronal model we have a positive coefficient $b$, which explains the restriction of our analysis to the case $b>0$.
\end{remark}

\section{The fractional-order Morris-Lecar neuronal model}

\subsection{Construction of the fractional-order model}

Neuronal activity of biological neurons has been typically modeled using the classical Hodgkin-Huxley mathematical model \cite{Hodgkin-Huxley}, dating back to 1952, including four nonlinear differential equations for the membrane potential and gating variables of ionic currents. Several lower dimensional simplified versions of the Hodgkin-Huxley model have been introduced in 1962 by Fitzhugh and Nagumo \cite{FitzHugh}, in 1981 by Morris and Lecar \cite{Morris_Lecar_1981} and in 1982 by Hindmarsh and Rose \cite{Hindmarsh-Rose-1}. These simplified models have an important advantage: while still being relatively simple, they allow for a good qualitative description of many different patterns of the membrane potential observed in experiments.

The classical Morris-Lecar neuronal model \cite{Morris_Lecar_1981} describes the oscillatory voltage patterns of Barnacle muscle fibers. Mathematically, the Morris-Lecar model is described by the following system of differential equations:
\begin{equation}\label{ML-model}
\left\{
  \begin{array}{rl}
C_m\frac{\mathrm{d}V}{\mathrm{d}t}=&g_{Ca}M_{\infty}(V)(V_{Ca}-V)+g_K N(V_K-V)+\\
&+g_L(V_L-V)+I\\
\frac{\mathrm{d}N}{\mathrm{d}t}=&\overline{\lambda_N}\cdot\lambda(V)(N_{\infty}(V)-N)
  \end{array}
\right.
\end{equation}
where $V$ is the membrane potential, $N$ is the gating variable for $K^+$, $C_m$ is the membrane capacitance, $I$ represents the externally applied current, $V_{Ca}$, $V_K$ and $V_L$ denote the equilibrium potentials for $Ca^{2+}$, $K^+$ the leak current and $g_{Ca}$, $g_K$ and $g_L$ are positive constants representing the maximum conductances of the corresponding ionic currents, and $\overline{\lambda_N}$ is the maximum rate constant for the $K^+$ channel opening.

The following assumptions are usually taken into consideration:
\begin{itemize}
  \item $M_\infty$ and $N_\infty$ are increasing functions of class $C^1$ defined on $\mathbb{R}$ with values in $(0,1)$;
  \item $\lambda$ is a positive function of class $C^1$ on $\mathbb{R}$;
  \item $V_K<V_L<0<V_{Ca}$.
\end{itemize}

The conductances of both $Ca^{2+}$ and $K^+$ are sigmoid functions with respect to the membrane voltage $V$. Particular functions considered previously in the literature, which satisfy the above assumptions, are:
\begin{align*}
M_{\infty}(V) &=\frac{1}{2}\left ( 1+\tanh\left (  \frac{V-V_1}{V_2} \right ) \right )\\
N_{\infty}(V) &=\frac{1}{2}\left ( 1+\tanh\left (  \frac{V-V_3}{V_4} \right ) \right )\\
\lambda(V) &=\cosh\left (  \frac{V-V_3}{V_4} \right )
\end{align*}
where $V_i$ are positive constants, $i\in\{1,2,3,4\}$.

In electrophysiological experiments, the neuronal membrane is considered to be equivalent to a resistor-capacitor circuit. In this context, based on experimental observations concerning biological neurons \cite{Anastasio,Lundstrom}, the fractional-order capacitor proposed by Westerlund and Ekstam \cite{Westerlund1994capacitor} has an utmost importance. They showed that Jacques Curie's empirical law for the current through capacitors and dielectrics leads to the following capacitive current-voltage relationship for a non-ideal capacitor:
$$I_c^{\alpha}=C_m^{\alpha}\frac{d^{\alpha}V_c}{dt^{\alpha}}$$
where $0<\alpha<1$, the fractional-order capacitance with units (amp/volt)sec$^\alpha$ is denoted by $C_m^{\alpha}$, and $\frac{d^{\alpha}}{dt^{\alpha}}$ represents a fractional-order differential operator \cite{Weinberg_2015}.

Several types of fractional-order neuronal models have been investigated in the recent years: fractional leaky integrate-and-fire model \cite{Teka2014neuronal}, fractional-order Hindmarsh-Rose model \cite{Jun_2014},  three-dimensional slow-fast fractional-order Morris-Lecar models \cite{Shi_2014,Upadhyay2016fractional} and fractional-order Hodgkin-Huxley models \cite{Teka2016power,Weinberg_2015}.

Starting from system (\ref{ML-model}), we first consider the following general fractional-order Morris-Lecar neuronal model with two fractional orders $p,q\in(0,1)$:
\begin{equation}\label{sistem1}
\left\{
  \begin{array}{rl}
C_m(q)\cdot ^c\!\!D^qV(t)=&g_{Ca}M_{\infty}(V)(V_{Ca}-V)+g_KN(V_K-V)+\\
&+g_L(V_L-V)+I\\
^cD^pN(t)=&\overline{\lambda_N}^p\cdot\lambda(V)(N_\infty(V)-N)
  \end{array}
\right.
\end{equation}
where $C_m(q)=\frac{\tau^q}{R_m}$ is the membrane capacitance \cite{Weinberg_2015}, $R_m$ is the membrane resistance, $\tau$ is the time constant. It is important to note that for $q=1$ we obtain the classical formula for the classical integer-order capacitance. We emphasize that the inclusion of the fractional-order capacitance to the left hand-side of the first equation is a straightforward answer to the dimensional consistency problem (units of measurement consistency) of system (\ref{sistem1}). For the same reason, in right side of the second equation we introduce the term $\overline{\lambda_N}^p$ (see for example \cite{Diethelm2013dengue} for a similar approach), and in this way, the dimensions of both sides coincide, and are (seconds)$^{-p}$.

For the theoretical investigation of the fractional order Morris-Lecar neuronal model, we nondimensionalize the system (\ref{sistem1}) as in Appendix B, with the substitutions:
$$v(t)=\frac{V(\tau t)}{V_{Ca}}\quad,\quad n(t)=N(\tau t).$$
Considering the following dimensionless constants:
\begin{align*}
&v_K=\frac{V_K}{V_{Ca}}, \quad v_L=\frac{V_L}{V_{Ca}},\quad v_i=\frac{V_i}{V_{Ca}},~~i\in\{1,2,3,4\}\\
&\gamma_x=R_m\cdot g_x,~~x\in\{Ca,K,L\},\quad\tilde{I}=R_m\cdot \frac{I}{V_{Ca}},
\end{align*}
and the functions
\begin{align*}
m_\infty(v)=&M_\infty(V_{Ca}v)=\frac{1}{2}\Big(1+\tanh\left(\frac{v-v_1}{v_2}\right)\Big),\\
n_{\infty}(v)=&N_\infty(V_{Ca}v)=\frac{1}{2}\left ( 1+\tanh\left (  \frac{v-v_3}{v_4} \right ) \right ),\\
\ell(v)=&\lambda(V_{Ca}v)=\cosh\Big(\frac{v-v_3}{2v_4}\Big),
\end{align*}
we obtain the following nondimensional fractional-order system
\begin{equation}\label{sistem2}
\left\{
\begin{array}{rl}
^c\!D^qv(t)=&\gamma_{Ca}m_\infty(v)(1-v)+\gamma_K\cdot n(v_K-v)+\\
&+\gamma_L(v_L-v)+\tilde{I}\\
^c\!D^pn(t)=&(\tau\overline{\lambda_N})^p\cdot \ell(v)(n_\infty(v)-n)
\end{array}\right.
\end{equation}
It is important to notice that, based on this procedure, the nondimensional system also involves the term $(\tau\overline{\lambda_N})^p$ in the right hand side of the second equation. Therefore, the correct version of the fractional-order variant of the Morris-Lecar neuronal model is quite different from the nondimensional version considered in \cite{Shi_2014}, where the fractional-order capacitance $C_m(q)$ appearing in the dimensional system has not been taken into account (nor the dimensional consistency problem) and equal fractional orders have been considered for both equations (i.e. $p=q$). In fact, it appears that the nondimensionalization process which had to be carried out to obtain the nondimensional system in \cite{Shi_2014} did not take into account the simple property presented in Proposition \ref{prop_aq}.

In fact, we have to remark that there is no known biological reason to consider a fractional order derivative for the gating variable, and therefore, in what follows, we consider $p=1$ as in \cite{Weinberg_2015}, obtaining the following nondimensional system:
\begin{equation}\label{sistem3}
\left\{
\begin{array}{rl}
^c\!D^qv(t)=&\gamma_{Ca}m_\infty(v)(1-v)+\gamma_K\cdot n(v_K-v)+\\
&+\gamma_L(v_L-v)+\tilde{I}\\
\dot{n}(t)=&\phi\ell(v)(n_\infty(v)-n)
\end{array}\right.
\end{equation}
where $\phi=\tau\cdot \overline{\lambda_N}$.

\subsection{Existence of equilibrium states}
System (\ref{sistem3}) is a particular case of the following generic two-dimensional fractional-order conductance-based neuronal model:
\begin{equation}\label{sistem4}
\left\{
\begin{array}{rl}
^c\!D^qv(t)&=I-I(v,n)\\
\dot{n}(t)&=\phi\ell(v)(n_\infty(v)-n)
\end{array}\right.
\end{equation}
where $v$ and $n$ represent the membrane potential and the gating variable of the neuron, $I$ is an externally applied current, $I(v,n)$ represents the ionic current, $\ell(v)$ and $n_\infty(v)$ are the rate constant for opening ionic channels and the fraction of open ionic channels at steady state, respectively.

In particular, for the Morris-Lecar fractional neuronal model (\ref{sistem3}), we have:
\begin{equation}\label{eq.I.ML}
  I(v,n)=\gamma_{Ca}m_\infty(v)(v-1)+\gamma_K\cdot n(v-v_K)+\gamma_L(v-v_L).
\end{equation}

The equilibrium states of system (\ref{sistem4}) are the solutions of the following algebraic system:
\begin{equation*}
\begin{cases}
&I=I(v,n)\\
&n=n_\infty(v)
\end{cases},
\end{equation*}
which is equivalent to
\begin{equation*}
\begin{cases}
&I=I(v,n_\infty(v)):=I_\infty(v)\\
&n=n_\infty(v)
\end{cases}.
\end{equation*}

In the following, we assume that the function $I_\infty(v)$ satisfies the following properties:
\begin{itemize}
  \item[(A1)] $I_\infty\in C^1(\mathbb{R})$;
  \item[(A2)] $\ds\lim\limits_{v\to -\infty}I_\infty(v)=-\infty$ and $\ds\lim\limits_{v\to\infty}I_\infty(v)=\infty$;
  \item[(A3)] $I'_{\infty}$ has exactly two real roots $v_\alpha<v_\beta$.
\end{itemize}
It is important to underline that these properties are satisfied in the particular case of the Morris-Lecar neuronal model with the function $I(v,n)$ given by (\ref{eq.I.ML}).

We denote $I_{max}=I_\infty(v_\alpha)$, $I_{min}=I_\infty(v_\beta)$. Then $I_\infty$ is increasing on the intervals $(-\infty, v_\alpha]$ and $[v_\beta, \infty)$ and decreasing on the interval $(v_\alpha, v_\beta)$.

As $I_\infty:(-\infty, v_\alpha]\to (-\infty, I_{max}]$ is increasing and continuous, it follows that it is bijective. We denote
$I_1=I_\infty|_{(-\infty, v_\alpha]}$ the restriction of $I_\infty$ to the interval $(-\infty,v_\alpha]$ and consider its inverse:
$$v_1:(-\infty, I_{max}]\to (-\infty, v_\alpha],\quad v_1(I)=I_1^{-1}(I).$$
Therefore, $(v_1(I),n_\infty(v_1(I)))$, with $I<I_{max}$, represents the first branch of equilibrium states of system (\ref{sistem4}).

We obtain the other two branches of equilibrium states in a similar way:
$$I_2=I_\infty|_{(v_\alpha,v_\beta)},\quad v_2:(I_{min},I_{max})\to (v_\alpha, v_\beta),\quad v_2(I)=I_2^{-1}(I)$$
$$I_3=I_\infty|_{[v_\beta, \infty)},\quad v_3:[I_{min}, \infty)\to [v_\beta, \infty),\quad v_3(I)=I_3^{-1}(I).$$

\begin{remark}
	We have the following situations:
	\begin{itemize}
		\item If $I<I_{min}$ or if $I>I_{max}$, then system (\ref{sistem4}) has an unique equilibrium state.
		\item If $I=I_{min}$ or if $I=I_{max}$, then system (\ref{sistem4}) has two equilibrium states.
		\item If $I\in (I_{min},I_{max})$, then system (\ref{sistem4}) has three equilibrium states.
	\end{itemize}
\end{remark}

\subsection{Stability of equilibrium states}
The Jacobian matrix associated to the system (\ref{sistem4}) at an equilibrium state $(v^\star,n^\star)=(v^\star, n_\infty(v^\star))$ is:
\begin{equation*}
J=\begin{bmatrix}
-I_v(v^\star,n^\star)& -I_n(v^\star,n^\star)\\
\phi\ell'(v^\star)[n_\infty(v^\star)-n^\star]+\phi\ell(v^\star)n'_\infty(v^\star)& -\phi\ell(v^\star))
\end{bmatrix}
\end{equation*}
Since $n_{\infty}(v^\star)=n^\star$, we have:
\begin{equation*}
J=\begin{bmatrix}
-I_v(v^\star,n_{\infty}(v^\star))& -I_n(v^\star,n_{\infty}(v^\star))\\
\phi\ell(v^\star)n'_\infty(v^\star)& -\phi\ell(v^\star))
\end{bmatrix}
\end{equation*}
Based on the considerations from section 3, the characteristic equation at the equilibrium state $(v^\star,n^\star)$ is
\begin{equation}\label{ec.car}
s^{q+1}+a(v^\star)s+b(v^\star)s^q+c(v^\star)=0
\end{equation}
where
\begin{align*}
a(v^\star)&=I_v(v^\star,n_\infty(v^\star)),\\
b(v^\star)&=\phi\ell(v^\star)>0,\\
c(v^\star)&=\det(J)=\\
&=\phi\ell(v^\star)[I_v(v^\star, n_\infty(v^\star))+n'_\infty(v^\star)\cdot I_n(v^\star, n_\infty(v^\star))]=\\
&=\phi\ell(v^\star)I'_\infty(v^\star).
\end{align*}

The equilibrium point $(v^\star,n^\star)=(v^\star, n_\infty(v^\star))$ is asymptotically stable if and only if all the roots of the characteristic equation (\ref{ec.car}) are in the left half-plane (i.e. $Re(s)<0$). In what follows, we will show how the theoretical results presented in section 3 can be applied to analyze the stability of the steady states of systems (\ref{sistem3}) and (\ref{sistem4}).

\begin{proposition}\label{prop.b2.unstable}
The second branch of equilibrium points \linebreak $(v_2(I),n_\infty(v_2(I)))$ (with $I\in(I_{min},I_{max})$) of system (\ref{sistem4}) is unstable, regardless of the fractional order $q$.
\end{proposition}

\begin{proof}
Let $I\in(I_{min},I_{max})$ and $v^\star=v_2(I)\in(v_\alpha,v_\beta)$. It follows that $I'_\infty(v^\star)<0$ and hence $c(v^\star)<0$. Based on Proposition \ref{prop122}, it follows that the characteristic equation \ref{ec.car} has at least one positive real root. Hence, the equilibrium point $(v^\star,n^\star)=(v_2(I),n_\infty(v_2(I)))$ is unstable, regardless of the fractional order $q$.\qed
\end{proof}

\begin{proposition}\label{prop.b13.stab}
If $(v^\star,n^\star)=(v^\star, n_\infty(v^\star))$ is an equilibrium point belonging to either the first or the third branch, such that $v^\star\notin\{v_\alpha,v_\beta\}$ and
$$a(v^\star)>-\min\{b(v^\star),1\},$$
then $(v^\star,n^\star)$ is asymptotically stable, regardless of the fractional order $q$.
\end{proposition}

\begin{proof}
If $(v^\star,n^\star)$ belongs to the first or third branch and $v^\star\notin\{v_\alpha,v_\beta\}$, we have $I'_\infty(v^\star)> 0$ which is equivalent to $c(v^\star)> 0$. Point 2.(b) of Corollary \ref{cor.stab.lin} implies that $(v^\star, n^\star)$ is asymptotically stable, regardless of the fractional order $q$.\qed
\end{proof}

\begin{remark}
We note that saddle-node bifurcations may take place in the generic system (\ref{sistem4}) if and only if $s=0$ is a root of the characteristic equation (\ref{ec.car}), which is equivalent to $c(v^\star)=0$, which in turn, means that $v^\star\in\{v_\alpha,v_\beta\}$. This obviously corresponds to the collision of two branches of equilibrium states (first two branches at $v_\alpha$ and last two branches at $v_\beta$, respectively).
\end{remark}

It is important to emphasize that Propositions \ref{prop.b2.unstable} and \ref{prop.b13.stab} have been obtained for the whole class of generic  fractional-order conductance-based neuronal models (\ref{sistem4}). In what follows, we restrict ourselves to the Morris-Lecar model (\ref{sistem3}). Additional information on the three branches of equilibrium states is given in Fig. \ref{fig.ramuri}

\begin{corollary}\label{prop.b13.ML} For the particular case of the Morris-Lecar model (\ref{sistem3}), with the function $I(v,n)$ given by (\ref{eq.I.ML}), assuming that:
$$v_K<v_\alpha<v_\beta<1,$$
we have:
\begin{enumerate}
\item Any equilibrium state $(v^\star, n^\star)$ of system (\ref{sistem3}) belonging to the first branch, with $v^\star\le v_{K}$, is asymptotically stable, regardless of the fractional order $q$.
\item Any equilibrium state $(v^\star, n^\star)$ of system (\ref{sistem3}) belonging to the third branch, with $v^\star\ge 1$, is asymptotically stable, regardless of the fractional order $q$.
\item Any equilibrium state $(v^\star, n^\star)$ of system (\ref{sistem3}) belonging to the second branch is unstable, regardless of the fractional order $q$.
\end{enumerate}
\end{corollary}

\begin{proof}
Let $(v^\star, n^\star)=(v^\star,n_\infty(v^\star))$ be an equilibrium state of system (\ref{sistem3}) belonging to the third branch, with $v^\star\ge 1$. As $I(v,n)$ is given by (\ref{eq.I.ML}), we have:
$$
I_v(v,n)=g_{Ca}[m'_{\infty}(v)(v-1)+m_{\infty}(v)]+g_k\cdot n+g_L
$$
and hence, we obtain:
\begin{align*}
a(v^\star)&=I_v(v^\star,n^\star)=\\
&=g_{Ca}[m'_{\infty}(v)(v^\star-1)+m_{\infty}(v^\star)]+g_k\cdot n_\infty(v^\star)+g_L\geq 0.
\end{align*}
Based on Proposition \ref{prop.b13.stab}, it follows that $(v^\star, n^\star)$ is asymptotically stable.

On the other hand, if $(v^\star, n^\star)=(v^\star,n_\infty(v^\star))$ is an equilibrium state of system (\ref{sistem3}) belonging to the first branch, with $v^\star\le v_{K}$, we first compute:
$$I'_{\infty}(v)=\frac{\mathrm{d}}{\mathrm{d}v}I(v,n_\infty(v))=I_v(v,n_\infty(v))+n_\infty'(v)\cdot I_n(v,n_\infty(v)),$$
and therefore:
$$
a(v^\star)=I_v(v^\star,n_\infty(v^\star))=n_\infty'(v^\star)\cdot g_k\cdot (v_K-v^\star)+I'_{\infty}(v^\star)\geq 0,
$$
due to the fact that $n_\infty$ is increasing on the whole real line and $I_\infty$ is increasing on $(-\infty,v_\alpha]$. Hence, based on Proposition \ref{prop.b13.stab}, it follows that $(v^\star, n^\star)$ is asymptotically stable.

The last part of the Proposition follows directly from Proposition \ref{prop.b2.unstable}.\qed
\end{proof}

\begin{remark}
In the following, we will discuss the stability of equilibrium states $(v^\star,n^\star)$ belonging to the first or third branch, with $v^\star\in(v_K,v_\alpha)$ or $v^\star\in(v_\beta,1)$, respectively.

Assume that $\phi$ is small (i.e. $\phi\ll 1$) and that $\ell(v)<\phi^{-1}$, for any $v\in(v_K,v_\alpha)\cup (v_\beta,1)$ (these are true in the case of numerical values considered in the literature). In this case, we have $b(v)=\phi\ell(v)<1$, for any $v\in(v_K,v_\alpha)\cup (v_\beta,1)$. Moreover, from the last part of the proof of Corollary \ref{prop.b13.ML}, we have
$$
a(v)=n_\infty'(v)\cdot g_k\cdot (v_K-v)+I'_{\infty}(v)
$$
and hence, we can easily see that $a(v_\alpha)<0$ and $a(v_\beta)<0$ (as $v_\alpha$ and $v_\beta$ are the roots of $I'_{\infty}$). On the other hand, from the proof of Corollary \ref{prop.b13.ML}, we know that $a(v_K)>0$ and $a(1)>0$, and therefore, the function $a(v)$ changes its sign on the intervals $(v_K,v_{\alpha})$ and $(v_{\beta},1)$, respectively. According to our assumption that $\phi$ is small, it follows that the function $a(v)+b(v)$ also changes its sign on each of the intervals $(v_K,v_{\alpha})$ and $(v_{\beta},1)$. Therefore, there exist two roots $v'\in(v_K,v_{\alpha}) $ and $v''\in(v_{\beta},1)$ of the function $a(v)+b(v)$. We will further assume that these roots are unique, which is in accordance with the numerical data. Based on Proposition \ref{prop.b13.stab}, we deduce that an equilibrium state $(v^\star,n^\star)$ belonging to the first branch or third branch with $v^\star<v'$ or $v^\star>v''$, respectively is asymptotically stable, regardless of the fractional order $q$ (see Fig. \ref{fig.ramuri}).

The stability of an equilibrium state $(v^\star,n^\star)$ belonging to the first branch with $v^\star\in[v',v_\alpha)$ depends on the fractional order $q$. Indeed, according to \ref{cor.stab.lin}, $(v^\star,n^\star)$ is $\mathcal{O}(t^{-q})$-asymptotically stable if and only if
$$a(v^\star)>a^\star(b(v^\star),c(v^\star),q),$$
where the function $a^\star$ is defined by (\ref{eq.a.star}) (see Fig \ref{fig.B1.q}).

At the critical value $q^\star$ defined implicitly by the equality
$$a(v^\star)=a^\star(b(v^\star),c(v^\star),q^\star),$$
a Hopf bifurcation is expected to occur, as it can be deduced from Proposition \ref{prop122}, points 3.(b,c). We emphasize that even though bifurcation theory in integer-order dynamical systems has been widely and rigorously studied (see for example
\cite{Kuznetsov}), at this time, in the case of fractional-order systems, very few theoretical results are known regarding bifurcation
phenomena. In \cite{ElSaka}, some conditions for the occurrence of Hopf bifurcations have been formulated, based on observations arising from
numerical simulations. Moreover, a center manifold theorem has been recently obtained in \cite{MaLi2016center}. However, the complete theoretical characterization of the Hopf bifurcation in fractional-order systems are still open questions. This is the reason why we rely on numerical simulations to assess the qualitative behavior of fractional order systems near a Hopf bifurcation point, as well as the stability of the resulting limit cycle. Bifurcations in the classical integer-order Morris-Lecar neuronal model are well-understood and have been thoroughly investigated in \cite{Tsumoto2006bifurcations}.

On the third branch, when analyzing the stability of an equilibrium state $(v^\star,n^\star)$ with $v^\star\in(v_\beta,v'']$, according to the numerical data, two situations may occur. Let us denote by $v'''\in(v_\beta,v'')$ the root of the equation
$$a(v)+b(v)+c(v)+1=0.$$
If $v^\star\in (v_\beta,v''']$, we have $a(v^\star)+b(v^\star)<0$ and $a(v^\star)+b(v^\star)+c(v^\star)+1\leq 0$, and therefore, from Corollary \ref{cor.stab.lin} point 2.(d), we obtain that $(v^\star,n^\star)$ is unstable, regardless of the fractional order $q$.

If $v^\star\in (v''',v'']$, the equilibrium state $(v^\star,n^\star)$ is $\mathcal{O}(t^{-q})$-asymptotically stable if and only if
$$a(v^\star)>a^\star(b(v^\star),c(v^\star),q),$$
where the function $a^\star$ is defined by (\ref{eq.a.star}) (see Fig \ref{fig.B3.q}).  As in the case of the first branch, at the critical value defined $q^\star$ defined implicitly by the equality
$$a(v^\star)=a^\star(b(v^\star),c(v^\star),q^\star),$$
a Hopf bifurcation is expected to occur.
\end{remark}

\begin{figure*}[htbp]
	\centering
	\includegraphics*[width=0.8\linewidth]{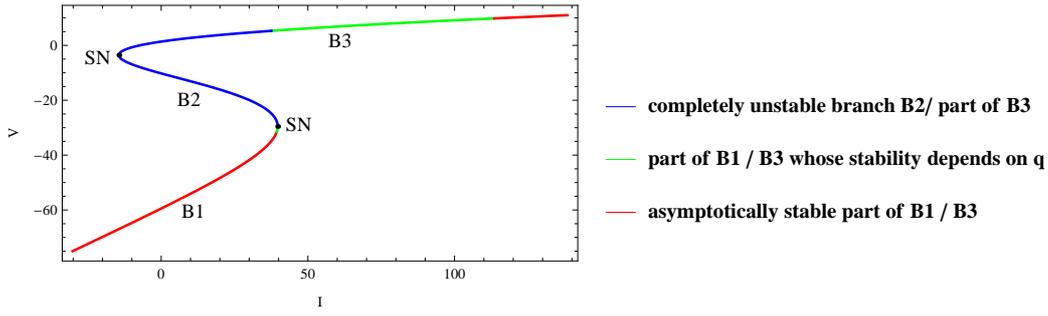}
\caption{Branches of equilibrium states for the Morris-Lecar model (\ref{sistem1}), with the parameter values given by Table \ref{tab1}. Here, $V'=-31.403$, $V_\alpha=-29.568$, $V_\beta=-3.5774$, $V'''=5.28457$ and $V''=9.82288$. The three branches coexist if and only if $I\in(-14.4204,39.6935)$.}
	\label{fig.ramuri}
\end{figure*}

\begin{figure}[htbp]
	\centering
	\includegraphics*[width=0.9\linewidth]{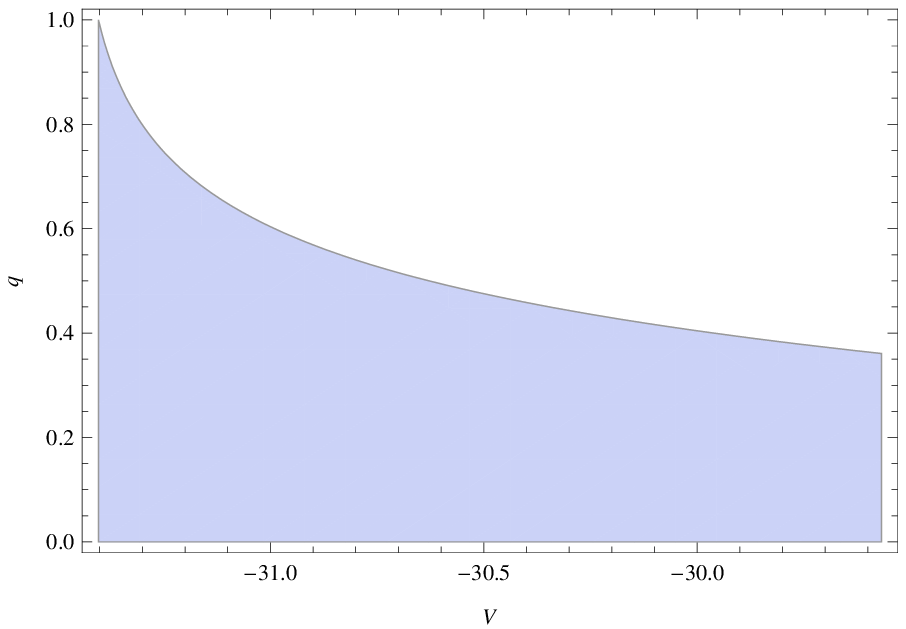}
	\caption{Stability of equilibrium states $(V^\star,N^\star)$ from branch B1 of the Morris-Lecar model (\ref{sistem1}) (with $p=1$) with $V^\star\in[V',V_\alpha)=[-31.403,-29.568)$ depends on the fractional order $q$. The blue curve represents the critical values of $q$ for which Hopf bifurcations may occur in a neighborhood of the corresponding equilibrium state $(V^\star,N^\star)$. The asymptotic stability region is the shaded area below the curve.}
	\label{fig.B1.q}
\end{figure}

\begin{figure}[htbp]
	\centering
	\includegraphics*[width=0.9\linewidth]{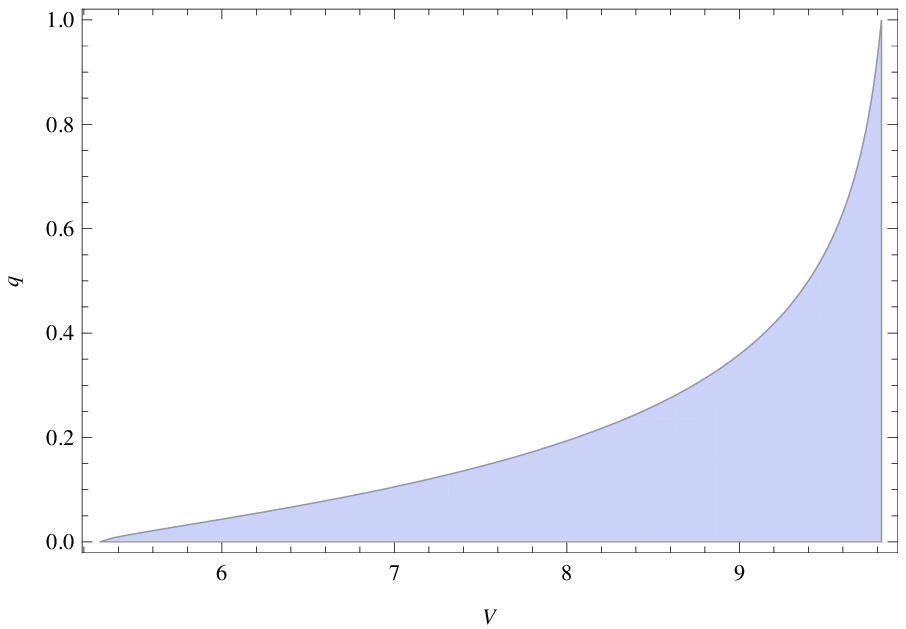}
	\caption{Stability of equilibrium states $(V^\star,N^\star)$ from branch B3 of the Morris-Lecar model (\ref{sistem1}) (with $p=1$) with $V^\star\in(V''',V'']=(5.28457,9.82288]$ depends on the fractional order $q$. The blue curve represents the critical values of $q$ for which Hopf bifurcations may occur in a neighborhood of the corresponding equilibrium state $(V^\star,N^\star)$. The asymptotic stability region is the shaded area below the curve.}
	\label{fig.B3.q}
\end{figure}

\subsection{Further numerical simulations}

In the numerical simulations, we use the numerical values given in Table \ref{tab1} for the parameters of system (\ref{sistem1}), corresponding to a type-I neuron \cite{Morris_Lecar_1981,Tsumoto2006bifurcations}.

\begin{table*}[htbp]
\caption{Numerical values and significance of the parameters used in the simulations.}
\label{tab1}
	\begin{center}
		\begin{tabular}{|c|c|c|l|}
			\hline
			\textbf{Parameter}&\textbf{Value}&\textbf{Unit}&\textbf{Significance}\\
			\hline
			$g_L$&$2$&$mmho/cm^2$&maximum or instantaneous values for leak pathways\\
			\hline
			$g_{Ca}$&$4$&$mmho/cm^2$&maximum or instantaneous values for $Ca^{++}$ pathways\\
			\hline
			$g_K$&$8$&$mmho/cm^2$&maximum or instantaneous values for $K^+$ pathways\\
			\hline
			$V_K$&$-80$&$mV$&equilibrium potential corresponding to $K^+$ conductances\\
			\hline
			$V_L$&$-60$&$mV$&equilibrium potential corresponding to leak conductances\\
			\hline
			$V_{Ca}$&$120$&$mV$&equilibrium potential corresponding to $Ca^{++}$ conductances\\
			\hline
			$V_1$&$-1.2$&$mV$&potential at which $M_{\infty}=0.5$\\
			\hline
			$V_2$&$18$&$mV$&reciprocal of slope of voltage dependence of $M_{\infty}$\\
			\hline
			$V_3$&$12$&$mV$&potential at which $N_{\infty}=0.5$\\
			\hline
			$V_4$&17.4&$mV$&reciprocal of slope of voltage dependence of $N_{\infty}$\\
			\hline
            $R_m$ &250& $\Omega\cdot cm^2$ & membrane resistance \\
            \hline
            $\tau$ &5& ms & time constant \\
            \hline
           $\overline{\lambda_N}$ & $1/15$ & s$^{-1}$ & maximum rate constant for the $K^+$ channel opening\\
           \hline
		\end{tabular}
	\end{center}
\end{table*}

Interesting spiking behavior can be observed by numerical simulations for the externally applied current of $I=40$ ($\mu A$) and different values of the fractional order $q$ (see Figs. \ref{fig.limit.cycles.q} and \ref{fig.spikes.q}). At $I=I_{max}=39.6935$, the first two branches of equilibrium states collide, at the saddle-node bifurcation point with abscissa $V_\alpha=-29.568$, and disappear for $I>I_{max}$.  When $I$ crosses the value $I_{max}$, the corresponding equilibrium states of branch B3 (with the abscissa slightly larger than $V'''$) are unstable for most values of the fractional order $q$ and asymptotically stable only for very small values of $q$ (which may be unrealistic from biologic point of view), as shown in Fig. \ref{fig.B3.q}. However, for $q$ large enough, a stable limit cycle exists in a neighborhood of each equilibrium state of branch B3 corresponding to $I$ slightly larger than $I_{max}$ (see green part of B3 in Fig. \ref{fig.ramuri}). Fig. \ref{fig.spikes.q} shows that for the same value of the externally applied current $I=40$ ($\mu A$), as the fractional order $q$ of the system decreases, the number of spikes over the same time interval increases, which may correspond to a better reflection of the biological properties by the fractional order model.

\begin{figure}[htbp]
	\centering
	\includegraphics*[width=\linewidth]{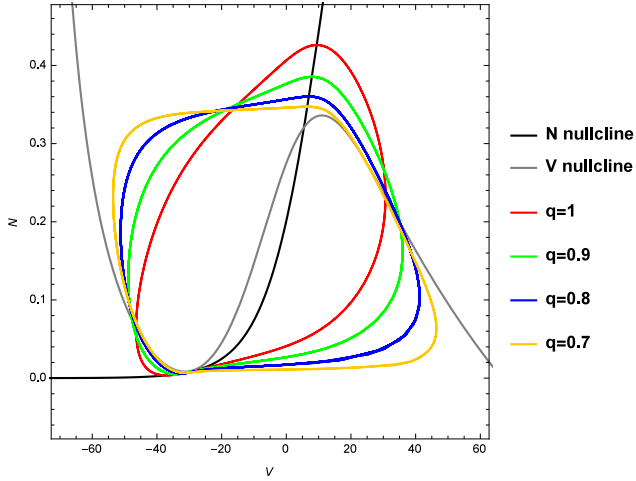}
\caption{Limit cycles for the fractional-order Morris-Lecar model (\ref{sistem1}) with various values of the fractional order $q$, when $I=40$.}
	\label{fig.limit.cycles.q}
\end{figure}

\begin{figure}[htbp]
	\centering
	\includegraphics*[width=\linewidth]{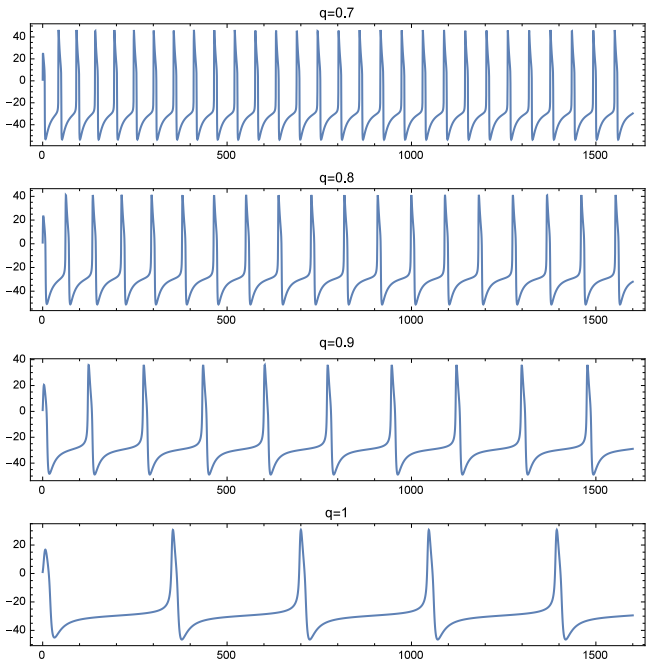}
	\caption{Evolution of $V$ with respect to time for the fractional-order Morris-Lecar model (\ref{sistem1}) with various values of the fractional order $q$, when $I=40$. The considered initial condition is the resting state.}
	\label{fig.spikes.q}
\end{figure}

\section{Conclusions}
In this paper, we have obtained necessary and sufficient conditions for the asymptotic stability of a two-dimensional incommensurate order linear autonomous system with one frac\-tional-order derivative and one first-order derivative. These theoretical results have been successfully applied to the investigation of the equilibrium states of a fractional-order Morris-Lecar neuronal model.

The extension of the methods presented in the first part of the paper to more complicated (higher dimensional) incommensurate order linear fractional order systems represents a direction for future research, possibly leading to more extensive generalizations of the classical Routh-Hurwitz stability conditions. A potential application of such results concerns the analysis of neuronal networks composed of several neurons of Morris-Lecar type.

\appendix

\section{Proof of Proposition \ref{prop_aq}.}
We compute:
	\begin{align*}
		^c\!D^qg(x)&=\frac{1}{\Gamma(-q)}\int\limits_{0}^{x}(x-t)^{-q-1}[g(t)-g(0)]dt\\
		&=\frac{1}{\Gamma(-q)}\int\limits_{0}^{x}(x-t)^{-q-1}[f(at)-f(0)]dt\\
		&=\frac{1}{a\Gamma(-q)}\int\limits_{0}^{ax}\Big(x-\frac{s}{a}\Big)^{-q-1}[f(s)-f(0)]dt\\
		&=\frac{1}{a\Gamma(-q)}\int\limits_{0}^{ax}a^{q+1}(ax-s)^{-q-1}[f(s)-f(0)]dt\\
		&=\frac{a^{q}}{\Gamma(-q)}\int\limits_{0}^{ax}(ax-s)^{-q-1}[f(s)-f(0)]dt\\
		&=a^q (^c\!D^qf)(ax)
	\end{align*}
It follows that:
	$$^c\!D^q g(x)=a^q\cdot ^c\!\!D^qf(ax),\quad\textrm{for any }a\neq 0,$$
which completes the proof.\qed

\section{Deduction of the nondimensional system (\ref{sistem3})}

Starting from system (\ref{sistem1}), we consider the substitutions
$$v(t)=kV(\alpha t)\quad,\quad n(t)=N(\alpha t),$$
where $\alpha$ and $k$ will be deduced in the following.

Applying Proposition \ref{prop122}, we have:
\begin{align*}
^c\!D^qv(t)=&k\cdot^c\!D^q[V(\alpha t)]\\
=&k\alpha^q(^c\!D^qV)(\alpha t)\\
=&k\alpha^q\frac{1}{C_m(q)}\Big[g_{Ca}M_\infty(V(\alpha t))(V_{Ca}-V(\alpha t))+\\
&+g_KN(\alpha t)(V_K-V(\alpha t))+g_L(V_L-V(\alpha t))+I\Big]\\
=&k\alpha^q\frac{1}{C_m(q)}\Big[g_{Ca}M_\infty\Big(\frac{v(t)}{k}\Big)\Big(V_{Ca}-\frac{v(t)}{k}\Big)+\\
&+g_K n(t)\Big(V_K-\frac{v(t)}{k}\Big)+g_L\Big(V_L-\frac{v(t)}{k}\Big)+I\Big]\\
=&\frac{\alpha^q}{C_m(q)}\Big[g_{Ca}m_\infty(v)(kV_{Ca}-v)+\\
&+g_Kn(kV_K-v)+g_L(kV_L-v)+kI\Big]
\end{align*}
and therefore, it makes sense to choose $k=\frac{1}{V_{Ca}}$.

Furthermore, with the notations from section 3, we obtain:
\begin{align*}
^c\!D^qv(t)=R_m\Big(\frac{\alpha}{\tau}\Big)^q\Big[&g_{Ca}m_\infty(v)(1-v)+g_K n(v_K-v)+\\
&+g_L(v_L-v)+\tilde{I}.
\end{align*}
At this step, it is easy to see that it makes sense to consider $\alpha=\tau$, which leads to:
\begin{align*}
^c\!D^qv(t)&=R_m\Big[g_{Ca}m_\infty(v)(1-v)+g_K n(v_K-v)+g_L(v_L-v)+\frac{I}{V_{Ca}}\Big]\\
&=\gamma_{Ca}m_\infty(v)(1-v)+\gamma_K\cdot n(v_K-v)+\gamma_L(v_L-v)+\tilde{I}.
\end{align*}
As for the second equation, applying Proposition \ref{prop122}, and taking into account that $\alpha=\tau$, we obtain:
\begin{align*}
^c\!D^pn(t)&=\alpha^p(^c\!D^p)(\alpha t)\\
&=(\alpha\overline{\lambda_N})^p\lambda(V(\alpha t))[N_\infty(V(\alpha t))-N(\alpha t)]\\
&=(\tau\overline{\lambda_N})^p\lambda\Big(\frac{v(t)}{k}\Big)\Big[N_\infty\Big(\frac{v(t)}{k}\Big)-n(t)\Big]\\
&=(\tau\overline{\lambda_N})^p\cdot\ell(v)[n_\infty(v)-n],
\end{align*}
Therefore, the nondimensional system (\ref{sistem3}) is found.

\bibliographystyle{spmpsci}

\begin{thebibliography}{10}
\providecommand{\url}[1]{{#1}}
\providecommand{\urlprefix}{URL }
\expandafter\ifx\csname urlstyle\endcsname\relax
  \providecommand{\doi}[1]{DOI~\discretionary{}{}{}#1}\else
  \providecommand{\doi}{DOI~\discretionary{}{}{}\begingroup
  \urlstyle{rm}\Url}\fi

\bibitem{Anastasio}
Anastasio, T.: The fractional-order dynamics of brainstem vestibulo-oculomotor
  neurons.
\newblock Biological Cybernetics \textbf{72}(1), 69--79 (1994)

\bibitem{Bonnet_2002}
Bonnet, C., Partington, J.R.: Analysis of fractional delay systems of retarded
  and neutral type.
\newblock Automatica \textbf{38}(7), 1133--1138 (2002)

\bibitem{Cermak2015}
{\v{C}}erm{\'a}k, J., Kisela, T.: Stability properties of two-term fractional
  differential equations.
\newblock Nonlinear Dynamics \textbf{80}(4), 1673--1684 (2015)

\bibitem{Chen_Lundberg}
Chen, J., Lundberg, K.H., Davison, D.E., Bernstein, D.S.: The final value
  theorem revisited: Infinite limits and irrational functions.
\newblock IEEE Control Systems Magazine \textbf{27}(3), 97--99 (2007)

\bibitem{Cottone}
Cottone, G., Paola, M.D., Santoro, R.: A novel exact representation of
  stationary colored gaussian processes (fractional differential approach).
\newblock Journal of Physics A: Mathematical and Theoretical \textbf{43}(8),
  085,002 (2010).
\newblock \urlprefix\url{http://stacks.iop.org/1751-8121/43/i=8/a=085002}

\bibitem{Datsko2012complex}
Datsko, B., Luchko, Y.: Complex oscillations and limit cycles in autonomous
  two-component incommensurate fractional dynamical systems.
\newblock Mathematica Balkanica \textbf{26}, 65--78 (2012)

\bibitem{Deng_2007}
Deng, W., Li, C., Lu, J.: Stability analysis of linear fractional differential
  system with multiple time delays.
\newblock Nonlinear Dynamics \textbf{48}, 409--416 (2007)

\bibitem{Diethelm_book}
Diethelm, K.: The analysis of fractional differential equations.
\newblock Springer (2004)

\bibitem{Diethelm2013dengue}
Diethelm, K.: A fractional calculus based model for the simulation of an
  outbreak of dengue fever.
\newblock Nonlinear Dynamics \textbf{71}(4), 613--619 (2013)

\bibitem{Doetsch}
Doetsch, G.: Introduction to the Theory and Application of the Laplace
  Transformation.
\newblock Springer-Verlag Berlin Heidelberg (1974)

\bibitem{ElSaka}
El-Saka, H., Ahmed, E., Shehata, M., El-Sayed, A.: On stability, persistence,
  and Hopf bifurcation in fractional order dynamical systems.
\newblock Nonlinear Dynamics \textbf{56}(1-2), 121--126 (2009)

\bibitem{Engheia}
Engheia, N.: On the role of fractional calculus in electromagnetic theory.
\newblock IEEE Antennas and Propagation Magazine \textbf{39}(4), 35--46 (1997)

\bibitem{FitzHugh}
FitzHugh, R.: Impulses and physiological states in theoretical models of nerve
  membrane.
\newblock Biophysical Journal \textbf{1}, 445--466 (1961)

\bibitem{Gorenflo_Mainardi}
Gorenflo, R., Mainardi, F.: Fractional calculus, integral and differential
  equations of fractional order.
\newblock In: A.~Carpinteri, F.~Mainardi (eds.) Fractals and Fractional
  Calculus in Continuum Mechanics, \emph{CISM Courses and Lecture Notes}, vol.
  378, pp. 223--276. Springer Verlag, Wien (1997)

\bibitem{Henry_Wearne}
Henry, B., Wearne, S.: Existence of turing instabilities in a two-species
  fractional reaction-diffusion system.
\newblock SIAM Journal on Applied Mathematics \textbf{62}, 870--887 (2002)

\bibitem{Heymans_Bauwens}
Heymans, N., Bauwens, J.C.: Fractal rheological models and fractional
  differential equations for viscoelastic behavior.
\newblock Rheologica Acta \textbf{33}, 210--219 (1994)

\bibitem{Hindmarsh-Rose-1}
Hindmarsh, J., Rose, R.: A model of the nerve impulse using two first-order
  differential equations.
\newblock Nature \textbf{296}, 162--164 (1982)

\bibitem{Hodgkin-Huxley}
Hodgkin, A., Huxley, A.: A quantitative description of membrane current and its
  application to conduction and excitation in nerve.
\newblock Journal of Physiology \textbf{117}, 500--544 (1952)

\bibitem{Jun_2014}
Jun, D., Guang-jun, Z., Yong, X., Hong, Y., Jue, W.: Dynamic behavior analysis
  of fractional-order Hindmarsh--Rose neuronal model.
\newblock Cognitive Neurodynamics \textbf{8}(2), 167--175 (2014)

\bibitem{Kaslik2012non}
Kaslik, E., Sivasundaram, S.: Non-existence of periodic solutions in
  fractional-order dynamical systems and a remarkable difference between
  integer and fractional-order derivatives of periodic functions.
\newblock Nonlinear Analysis: Real World Applications \textbf{13}(3),
  1489--1497 (2012)

\bibitem{Kilbas}
Kilbas, A., Srivastava, H., Trujillo, J.: Theory and Applications of Fractional
  Differential Equations.
\newblock Elsevier (2006)

\bibitem{Kuznetsov}
Kuznetsov, Y.A.: Elements of applied bifurcation theory, vol. 112.
\newblock Springer Science \& Business Media (2004)

\bibitem{Lak}
Lakshmikantham, V., Leela, S., Devi, J.V.: Theory of fractional dynamic
  systems.
\newblock Cambridge Scientific Publishers (2009)

\bibitem{Li_Ma_2013}
Li, C., Ma, Y.: Fractional dynamical system and its linearization theorem.
\newblock Nonlinear Dynamics \textbf{71}(4), 621--633 (2013)

\bibitem{Li-survey}
Li, C., Zhang, F.: A survey on the stability of fractional differential
  equations.
\newblock The European Physical Journal - Special Topics \textbf{193}, 27--47
  (2011)

\bibitem{Li_Chen_Podlubny}
Li, Y., Chen, Y., Podlubny, I.: Mittag-Leffler stability of fractional order
  nonlinear dynamic systems.
\newblock Automatica \textbf{45}(8), 1965 -- 1969 (2009)

\bibitem{Lundstrom}
Lundstrom, B., Higgs, M., Spain, W., Fairhall, A.: Fractional differentiation
  by neocortical pyramidal neurons.
\newblock Nature Neuroscience \textbf{11}(11), 1335--1342 (2008)

\bibitem{MaLi2016center}
Ma, L., Li, C.: Center manifold of fractional dynamical system.
\newblock Journal of Computational and Nonlinear Dynamics \textbf{11}(2),
  021,010 (2016)

\bibitem{Mainardi_1996}
Mainardi, F.: Fractional relaxation-oscillation and fractional phenomena.
\newblock Chaos Solitons Fractals \textbf{7}(9), 1461--1477 (1996)

\bibitem{Matignon}
Matignon, D.: Stability results for fractional differential equations with
  applications to control processing.
\newblock In: Computational Engineering in Systems Applications, pp. 963--968
  (1996)

\bibitem{Metzler}
Metzler, R., Klafter, J.: The random walk's guide to anomalous diffusion: a
  fractional dynamics approach.
\newblock Physics Reports \textbf{339}(1), 1 -- 77 (2000)

\bibitem{Morris_Lecar_1981}
Morris, C., Lecar, H.: Voltage oscillations in the barnacle giant muscle fiber.
\newblock Biophysical Journal \textbf{35}(1), 193 (1981)

\bibitem{Odibat2010}
Odibat, Z.M.: Analytic study on linear systems of fractional differential
  equations.
\newblock Computers \& Mathematics with Applications \textbf{59}(3), 1171--1183
  (2010)

\bibitem{Petras2008stability}
Petras, I.: Stability of fractional-order systems with rational orders.
\newblock arXiv preprint arXiv:0811.4102  (2008)

\bibitem{Podlubny}
Podlubny, I.: Fractional differential equations.
\newblock Academic Press (1999)

\bibitem{Radwan2008fractional}
Radwan, A.G., Elwakil, A.S., Soliman, A.M.: Fractional-order sinusoidal
  oscillators: design procedure and practical examples.
\newblock IEEE Transactions on Circuits and Systems I: Regular Papers
  \textbf{55}(7), 2051--2063 (2008)

\bibitem{Rivero2013stability}
Rivero, M., Rogosin, S.V., Tenreiro~Machado, J.A., Trujillo, J.J.: Stability of
  fractional order systems.
\newblock Mathematical Problems in Engineering \textbf{2013} (2013)

\bibitem{Sabatier2012stability}
Sabatier, J., Farges, C.: On stability of commensurate fractional order
  systems.
\newblock International Journal of Bifurcation and Chaos \textbf{22}(04),
  1250,084 (2012)

\bibitem{Shi_2014}
Shi, M., Wang, Z.: Abundant bursting patterns of a fractional-order
  Morris--Lecar neuron model.
\newblock Communications in Nonlinear Science and Numerical Simulation
  \textbf{19}(6), 1956--1969 (2014)

\bibitem{Sugimoto}
Sugimoto, N.: Burgers equation with a fractional derivative; hereditary effects
  on nonlinear acoustic waves.
\newblock Journal of Fluid Mechanics \textbf{225}, 631--653 (1991)

\bibitem{Teka2014neuronal}
Teka, W., Marinov, T.M., Santamaria, F.: Neuronal spike timing adaptation
  described with a fractional leaky integrate-and-fire model.
\newblock PLoS Comput Biol \textbf{10}(3), e1003,526 (2014)

\bibitem{Teka2016power}
Teka, W., Stockton, D., Santamaria, F.: Power-law dynamics of membrane
  conductances increase spiking diversity in a Hodgkin-Huxley model.
\newblock PLoS Comput Biol \textbf{12}(3), e1004,776 (2016)

\bibitem{Trachtler2016}
Tr{\"a}chtler, A.: On BIBO stability of systems with irrational transfer
  function.
\newblock arXiv:1603.01059  (2016)

\bibitem{Tsumoto2006bifurcations}
Tsumoto, K., Kitajima, H., Yoshinaga, T., Aihara, K., Kawakami, H.:
  Bifurcations in Morris--Lecar neuron model.
\newblock Neurocomputing \textbf{69}(4), 293--316 (2006)

\bibitem{Upadhyay2016fractional}
Upadhyay, R.K., Mondal, A., Teka, W.W.: Fractional-order excitable neural
  system with bidirectional coupling.
\newblock Nonlinear Dynamics pp. 1--15 (2016)

\bibitem{Wang2016stability}
Wang, Z., Yang, D., Zhang, H.: Stability analysis on a class of nonlinear
  fractional-order systems.
\newblock Nonlinear Dynamics \textbf{86}(2), 1023--1033 (2016)

\bibitem{Weinberg_2015}
Weinberg, S.H.: Membrane capacitive memory alters spiking in neurons described
  by the fractional-order Hodgkin-Huxley model.
\newblock PloS one \textbf{10}(5), e0126,629 (2015)

\bibitem{Westerlund1994capacitor}
Westerlund, S., Ekstam, L.: Capacitor theory.
\newblock IEEE Transactions on Dielectrics and Electrical Insulation
  \textbf{1}(5), 826--839 (1994)

\end{thebibliography}

\end{document}